\documentclass{amsart}
\usepackage{amsfonts,amssymb,amscd,amsmath,enumerate,verbatim,calc}
\usepackage[all]{xy}

\newcommand{\CM}{Cohen-Macaulay}

\newcommand{\wrt}{with respect to}

\newcommand{\m}{\mathfrak{m} }

\newcommand{\Z}{\mathbb{Z}  }

\newcommand{\rt}{\rightarrow}

\newcommand{\ov}{\overline}

\newcommand{\F}{\mathbb{F} }
\newcommand{\Fc}{\mathcal{F} }

\newcommand{\G}{\mathbb{G} }

\newcommand{\image}{\operatorname{image}}

\newcommand{\ann}{\operatorname{ann}}

\newcommand{\ini}{\operatorname{in}}

\newcommand{\rank}{\operatorname{rank}}
\newcommand{\coker}{\operatorname{coker}}

\newcommand{\Syz}{\operatorname{Syz}}

\newcommand{\projdim}{\operatorname{projdim}}

\newcommand{\Fitt}{\operatorname{Fitt}}

\theoremstyle{plain}

\newtheorem{theorem}{Theorem}[section]

\newtheorem{lemma}[theorem]{Lemma}
\newtheorem{proposition}[theorem]{Proposition}

\theoremstyle{definition}

\newtheorem{remark}[theorem]{Remark}
\newtheorem{example}[theorem]{Example}

\newtheorem{construction}[theorem]{Construction}
\theoremstyle{remark}

\begin{document}

\title[Pure resolutions]{On associated graded modules having a pure resolution}
 \author{Tony J. Puthenpurakal}
\date{\today}
\address{Department of Mathematics, Indian Institute of Technology Bombay, Powai, Mumbai 400 076, India}
\email{tputhen@math.iitb.ac.in}
\begin{abstract}
Let $A = K[[X_1,\cdots,X_n]]$ and let $\m = (X_1,\cdots,X_n)$. Let $M$ be a \CM \ $A$-module of codimension $p$. In this paper we give a necessary and sufficient condition for the associated graded module $G_\m(M)$ to have a pure resolution over the polynomial ring $G_\m(A) \cong K[X_1,\cdots,X_n]$.
\end{abstract}

\maketitle

\section{introduction}
Let $R = K[X_1,\cdots, X_n]$ and let $M$ be a finitely generated graded $R$-module of projective dimension $p$. Recall that $M$ has a \textit{pure resolution } of type $(d_0,d_1,\cdots, d_p)$ if the minimal resolution of $M$ is of the following form:
\[
0                                  \rt R(-d_p)^{\beta_p}  \rt   \cdots     \rt  R(-d_2)^{\beta_2}  \rt              R(-d_1)^{\beta_1}  \rt R(-d_0)^{\beta_0} \rt 0.
\]
Herzog and K\"{u}hl showed in \cite{HK} that the Betti numbers of a pure resolution of a \CM \
algebra are determined by its type and Huneke and Miller computed in \cite{HM} the multiplicity of
such an algebra, also in terms of its type. These two results can be extended to \CM \
$R$-modules, see \cite[p.\ 88]{BS}.  In the beautiful paper \cite{BS},   Boig  and S\"{o}derberg conjectured that 
the Betti diagram of any \CM \  $R$-module is a non-negative
linear combination of pure diagrams; furthermore, any pure diagram is a rational multiple of
the Betti diagram of some \CM \  $R$-module. This conjecture was proved in \cite{ES}. 

Let $A = K[[X_1,\ldots,X_n]]$ and let $\m = (X_1,\cdots,X_n)$.  Let $G_\m(A)$ be the associated graded ring of $A$ \wrt \ $\m$, i.e., 
 $G_\m(A) = \bigoplus_{n \geq 0} \m^n/\m^{n+1}$.  It is well-known that $G_\m(A) \cong K[X_1,\cdots,X_n]$. Let $M$ be a finitely generated $A$-module and let $G_\m(M) = \bigoplus_{n\geq 0} \m^n M/\m^{n+1} M$ be the associated graded module of $M$ with respect to $\m$. A natural question is when does $G_\m(M)$ have a pure resolution? To answer this question we construct a pure complex attached to $M$ as follows:
\begin{construction}
(1). Let $\phi \colon A^n \rt A^m$ be a non-zero $A$-linear map. Assume $\image \phi \subseteq \m^s A^m$ but $\image \phi \nsubseteq \m^{s+1} A^m$ for some $s \geq 0$. Let $\phi = (a_{ij})$ where $a_{ij} \in A$. By assumption $a_{ij} \in \m^s$.  It follows that
\[
\phi = \sum_{j \geq s} \phi_j
\]
where $\phi_j$ is a matrix with entries homogeneous forms of degree $j$.
Set $\ini(\phi) = \phi_s$.
We call $\ini(\phi)$ to be the \emph{initial form} of $\phi$. Set   $ v(\phi) = s $; the \emph{order} of $\phi$. 

\begin{example}
Let 
\[
\phi = \left(
       \begin{matrix}
       X_1^2 + X_2^2X_3 & 0 & X_3^2 \\
       0 & X_2^4 & X_1X_3 + X_2^3 \\
       X_1^3 & 0 & X_1X_2X_3 + X_2^4
       \end{matrix}
       \right).
\]
Then 
\[
\ini(\phi) =  \left(
       \begin{matrix}
       X_1^2  & 0 & X_3^2 \\
       0 & 0 & X_1X_3  \\
      0 & 0 & 0
       \end{matrix}
       \right).
\]
Also $v(\phi) = 2$. 
\end{example}

(2).  Let  
$$\F \colon 0 \rt F_p \xrightarrow{\phi_p} F_{p-1} \rt \cdots \rt F_2 \xrightarrow{\phi_2}  F_1 \xrightarrow{\phi_1}  F_0 \rt 0,$$ 
be a minimal resolution of $M$. Set $c_i = v(\phi_i)$ for $i = 1, 2, \cdots, p$. Let   $d_i = \sum_{j = 1}^{i}c_j$  for $i = 1, 2, \cdots, p$.  Let $\beta_i  = \rank F_i$ be the $i^{th}$ betti-number of $M$. It is easily shown  (see 2.3) that
we have a complex, $\ini(\F)$,
\[
   0 \rt R(-d_p)^{\beta_p} \xrightarrow{\ini(\phi_{p})} R(-d_{p-1})^{\beta_{p-1}} \rt
\cdots \rt R(-d_1)^{\beta_1} \xrightarrow{\ini(\phi_{1})}
R^{\beta_0} \rt 0.
\]
We also have an augumentation map $\epsilon \colon \ini(\F) \rt G_\m(M)$, i.e., an $R$-linear map $\epsilon \colon R^{\beta_0} \rt G_\m(M)$ such that $\epsilon \circ \ini(\phi_1) = 0$. Furthermore $\epsilon$ is surjective.
If $\G$ is another minimal resolution of $M$ then it can be easily shown that $\ini(\F) \cong \ini(\G)$ as  augumented complexes, see 2.3.
\end{construction}

Our first result is:
\begin{theorem}\label{first}
Let $M$ be a finitely generated $A$-module. Assume $G_\m(M)$ has a pure resolution. Let $\F$ be a minimal free resolution of $M$. Then $\ini(\F)$ is a minimal free resolution of $G_\m(M)$.
\end{theorem}

Theorem \ref{first} gives a necessary condition for $G_\m(M)$ to have a pure resolution.
The following result gives a sufficient condition for $G_\m(M)$ to have a pure resolution.

\begin{theorem}\label{pure}
Let $M$ be a \CM \ $A$-module and let $p = \projdim M$. Let $\beta_i = \beta_i(M)$.
Let $\F$ be a minimal resolution of $M$.  The following conditions are equivalent
\begin{enumerate}[\rm (i)]
\item
$G_\m(M)$ has a pure resolution.
\item
The following hold 
\begin{enumerate}[\rm (a)]
\item
$\ini(\F)$ is acyclic.
\item
For $i \geq 1$,
\[
\beta_i  =  (-1)^{i+1}\beta_0 \prod_{\substack{1 \leq j \leq p \\  j\neq i}} \frac{d_j}{d_j - d_i}.
\]
\item
The multiplicity of $M$,
\[
e_0(M)  =  \frac{\beta_0}{p!} \prod_{i = 1}^{p}d_i.
\]
\end{enumerate}
\end{enumerate}
\end{theorem}

If $M  = A /I$  then using a computer algebra program one can find  $G_\m(M)$. However if $M$ is not a cyclic module there is no computer algebra program to find a presentation of
$G_\m(M)$ as a $R$-module.  My motivation for this work was to find non-trivial examples of  a minimal resolution of $G_\m(M)$. Even when $p = \projdim M  = 1,2$ this is a non-trivial problem even when $G_\m(M)$ is \CM. 
  
Here is an overview of the contents of the paper. In section two we give our construction of the complex $\ini(\F)$. In section three we prove Theorem \ref{first}. In the next section we prove Theorem \ref{pure}. Finally in section five we give an non-trivial example of $G_\m(M)$ having a pure resolution.

\section{Construction of the complex $\ini(\F)$}
Throughout $A = k[[X_1,\ldots,X_n]]$ and let  $\m $ be the maximal ideal of $A$. All $A$-modules considered  will be finitely generated.
Let $R = G_\m(A) = \bigoplus_{n \geq 0} \m^n/\m^{n+1}$ be the associated graded ring of $A$. It is well-known that $R \cong k[X_1,\ldots,X_n]$. 
Let $M$ be an $A$-module and let $G_\m(M) = \bigoplus_{n \geq 0}\m^{n}M/\m^{n+1}M$ be the associated graded module of $M$. Clearly $G_\m(M)$ is a finitely generated $R$-module. 
Let
$\F$ be a minimal resolution of $M$. In this section we construct our complex $\ini(\F)$ of $R$-modules. We also show that there is an augumentation $\epsilon \colon \ini(\F) \rt  G_\m(M)$ with $\epsilon$ surjective.

\begin{proposition}\label{basic-prop}
Let $F_2 \xrightarrow{\phi_2} F_1 \xrightarrow{\phi_1} F_0$ be a complex of free $A$-modules. Let $v(\phi_1) = c_1$ and $v(\phi_2) = c_2$. Set $d_1 = c_1$ and $d_2 = c_1 + c_2$.  Then 
\begin{enumerate}[\rm(1)]
\item
we have a complex of $R = G_\m(A)$-modules
\[
G_\m(F_2)(-d_2) \xrightarrow{\ini(\phi_2)} G_\m(F_1)(-d_1) \xrightarrow{\ini(\phi_1)}
G_\m(F_0).
\]
\item
Let $M = \coker \phi_1$. Consider the natural map $\epsilon \colon G_\m(F_0) \rt G_\m(M)$. Then
\begin{enumerate}[\rm (a)]
\item
$\epsilon \circ \ini(\phi_1) = 0$.
\item
$\epsilon$ is surjective.
\end{enumerate}
\end{enumerate}
\end{proposition}
\begin{proof}
If $E$ is an $A$-module then set $\m^i E = E$ for $i \leq 0$. Notice  that for all $i \in \Z$ we have an complex
\[
 \m^{i-d_2}F_2 \rt \m^{i-d_1}F_1 \rt \m^{i}F_0 \rt \m^i M \rt 0.
\]
After tensoring with $A/\m$ and collecting terms 
 we have a complex
\[
G_\m(F_2)(-d_2) \xrightarrow{\alpha} G_\m(F_1)(-d_1) \xrightarrow{\beta}
G_\m(F_0) \xrightarrow{\epsilon} G_\m(M) \rt 0. 
\]
Observe that $\alpha = \ini(\phi_2)$ and $\beta = \ini(\phi_1)$. Also clearly $\epsilon$ is surjective.
\end{proof}

Next we show
\begin{proposition}\label{comp}
Let $\F \colon F_2 \xrightarrow{\phi_2} F_1 \xrightarrow{\phi_1} F_0$  and  $\G \colon G_2 \xrightarrow{\psi_2} F_1 \xrightarrow{\psi_1} F_0$  be  complexes of free $A$-modules. Suppose we have a commutative diagram
\[
\xymatrixrowsep{3pc}
\xymatrixcolsep{2.5pc}
\xymatrix{
     F_2 \ar@{->}[r]^{\phi_2}
     \ar@{->}[d]^{\theta_2}
&    F_1 \ar@{->}[r]^{\phi_1}
     \ar@{->}[d]^{\theta_1}
&   F_0    \ar@{->}[d]^{\theta_0}
\\
 G_2 \ar@{->}[r]^{\psi_2}
&    G_1 \ar@{->}[r]^{\psi_1}
&   G_0   
}
\]
such that $\theta_i$ are isomorphism's for $i = 0, 1, 2$.
Then
\begin{enumerate}[\rm (1)]
\item
$v(\phi_i) = v(\psi_i)$ for $i = 1, 2$.
\item
Let $v(\phi_1) = c_1$ and $v(\phi_2) = c_2$. Set $d_1 = c_1$ and $d_2 = c_1 + c_2$. 
Then we have a commutative diagram of $R = G_\m(A)$-modules
\[
\xymatrixrowsep{3pc}
\xymatrixcolsep{2.5pc}
\xymatrix{
     G_\m(F_2)(-d_2) \ar@{->}[r]^{\ini(\phi_2)}
     \ar@{->}[d]^{\ini(\theta_2)}
&    G_\m(F_1)(-d_1) \ar@{->}[r]^{\ini(\phi_1)}
     \ar@{->}[d]^{\ini(\theta_1)}
&   G_\m(F_0)    \ar@{->}[d]^{\ini(\theta_0)}
\\
 G_\m(G_2)(-d_2) \ar@{->}[r]^{\ini(\psi_2)}
&   G_\m(G_1)(-d_1) \ar@{->}[r]^{\ini(\psi_1)}
&   G_\m(G_0)   
}
\]
such that 
\begin{enumerate}[\rm (a)]
\item
The rows are complexes of $R$-modules.
\item
$\ini(\theta_i)$ is an isomorphism for $i = 0,1,2$.
\end{enumerate}
\item 
Set $M = \coker \phi_1$ and $M^\prime = \coker \phi_2$ and let $\xi \colon M \rt M^\prime$ be the isomorphism induced by the above commutative diagram.  Then we have a commutative diagram
\[
\xymatrixrowsep{3pc}
\xymatrixcolsep{2.5pc}
\xymatrix{
     G_\m(F_0) \ar@{->}[r]^{\epsilon}
     \ar@{->}[d]^{\ini(\theta_0)}
&   G_\m(M)    \ar@{->}[d]^{G_\m(\xi)} \ar@{->}[r]
&0
\\
 G_\m(G_0) \ar@{->}[r]^{\epsilon^\prime}
&   G_\m(M^\prime)  \ar@{->}[r]
&0 
}
\]
\end{enumerate} 
\end{proposition}
\begin{proof}
(0) Let $\delta \colon A^m \rt A^m$ be invertible $A$-linear map. We write $\delta = \sum_{j\geq 0} \delta_j$ where $\delta_j$ is a $m\times m$ matrix of forms of degree $j$. As $\delta \otimes A/\m \colon k^m \rt k^m$ is an isomorphism we get that $\delta_0 $ is an invertible matrix. Also notice $\delta_0 = \ini(\delta)$.

(1) and (2): Let $r = v(\phi_1)$. We write $\phi_1 = \sum_{j\geq r} \phi_{1,j}$ where $\phi_{1,j}$ is a matrix of forms of degree $j$. Let $s = v(\psi_1)$. Write $\psi_1 = \sum_{j\geq s} \psi_{1,j}$ where $\psi_{1,j}$ is a matrix of forms of degree $j$.
For $i = 0, 1$ write $\theta_i = \sum_{j\geq 0} \theta_{i,j}$ as before. As 
$\theta_0 \circ \phi_1 = \psi_1 \circ \theta_1$ we get that $\theta_{0,0} \circ \phi_{1,r} = \psi_{1,s}\circ \theta_{1,0}$. By (0) we get that $\theta_{0,0}$ and $\theta_{1,0}$
are invertible matrices of constants. So $r = s$. We also get 
$\ini(\theta_0) \circ \ini(\phi_1) = \ini(\psi_1)\circ \ini(\theta_1)$.
Similarly we get $v(\phi_2) = v(\psi_2)$ and 
$\ini(\theta_1) \circ \ini(\phi_2) = \ini(\psi_2)\circ \ini(\theta_2)$.

2(a) This follows from Proposition \ref{basic-prop}.

2(b) This follows from (0).

(3) This is obvious.
\end{proof}

\begin{construction}
 Let  
$$\F \colon 0 \rt F_p \xrightarrow{\phi_p} F_{p-1} \rt \cdots \rt F_2 \xrightarrow{\phi_2}  F_1 \xrightarrow{\phi_1}  F_0 \rt 0,$$ 
be a minimal resolution of $M$. Set $c_i = v(\phi_i)$ for $i = 1, 2, \cdots, p$. Let   $d_i = \sum_{j = 1}^{i}c_j$  for $i = 1, 2, \cdots, p$.  Let $\beta_i  = \rank F_i$ be the $i^{th}$ betti-number of $M$. Note $R^\beta_i = G_\m(F_i)$.  By Proposition \ref{basic-prop} it follows that 
we have a complex, $\ini(\F)$,
\[
   0 \rt R(-d_p)^{\beta_p} \xrightarrow{\ini(\phi_{p})} R(-d_{p-1})^{\beta_{p-1}} \rt
\cdots \rt R(-d_1)^{\beta_1} \xrightarrow{\ini(\phi_{1})}
R^{\beta_0} \rt 0.
\]
By Proposition \ref{basic-prop}  we also have an augumentation map $\epsilon \colon \ini(\F) \rt G_\m(M)$, i.e., an $R$-linear map $\epsilon \colon R^{\beta_0} \rt G_\m(M)$ such that $\epsilon \circ \ini(\phi_1) = 0$. Furthermore $\epsilon$ is clearly surjective.  If $\G$ is another minimal resolution of $M$ then by Proposition \ref{comp} it follows that we have an isomorphism of augumented complexes $\ini(\F)$ and $\ini(\G)$.
\end{construction}
\section{Proof of Theorem \ref{first}}
In this section we give a proof of Theorem \ref{first}.
We first prove:
\begin{lemma}\label{f-l}
Let $F_1 \xrightarrow{\phi} F_0 \rt M \rt 0$ be part of a minimal resolution of $M$. Assume that the minimal resolution of $G_\m(M)$ has the following form $\cdots \rt R^{a}(-s) \rt G_\m(F_0) \rt G_\m(M) \rt 0$.  Set $N = \image \phi$ and $\Fc = \{ N_i = \m^iF_0 \cap N \}_{i \in \Z}$. Then
\begin{enumerate}[\rm (1)]
\item
$N_i  = \m^{i-s}N$ for all $i \geq 0$.
\item
$\rank F_1 = a$.
\item
$v(\phi) = s$.
\item
The sequence
$G_\m(F_1)(-s) \xrightarrow{\ini(\phi)} G_\m(F_0) \rt G_\m(M) \rt 0$  can be extended to a minimal resolution of $G_\m(M)$.
\item
$\image \ini(\phi) \cong G_\m(N)(-s)$.
\end{enumerate}
\end{lemma}
\begin{proof}
(1) It is well-known that $\Fc$ is an $\m$-stable 
filtration on $N$. Furthermore we have an exact sequence
\begin{equation*}
0 \rt G_{\Fc}(N) \rt G_\m(F_0) \rt G_\m(M) \rt 0. \tag{*}
\end{equation*}
By our assumption it follows that $G_\Fc(N)$ is generated in degree $s$. So we have $N_s = N$ and $N_{s+j} = \m^j N + N_{s+j+1}$ 
for all $j \geq 1$. As $\Fc$ is $\m$-stable there exists $j_0$ such that $N_{s + j + 1} = \m N_{s+j}$ for all $j \geq j_0$. Fix $j \geq j_0$. Then  $N_{s+j} = \m^j N + N_{s+j+1} = \m^l N + \m N_{s+j}$. So by Nakayama Lemma $N_{s+j} = \m^jN$. 
We now show by descending induction that $N_{s+j} = \m^j N$ for all $j \leq j_0$. This is true for $j = j_0$ by the previous argument. Assume $N_{s+j+1} = \m^{j+1} N$ for some $j \leq j_0 -1$. Then notice
\[
\m^{j+1}N \subseteq \m N_{s+j} \subseteq N_{s+j+1}  = \m^{j+1}N.
\]
So we have $N_{s+j+1} = \m N_{s+j}$. As $N_{s+j} = \m^j N + N_{s+j+1} = \m^j N + \m N_{s+j}$, by Nakayama Lemma we get that
$N_{s+j} = \m^j N$.

(2) As $R^a(-s) \rt G_\m(N)(-s) \rt 0$ is minimal we have $a = \mu(G_\m(N)) = \mu(N) = \rank F_1$.

(3) Set $r = v(\phi)$. By \ref{basic-prop} we have a complex
\[
G_\m(F_1)(-r) \xrightarrow{\ini(\phi)} G_\m(F_0) \xrightarrow{\epsilon} G_\m(M) \rt 0.
\]
So $\ker \epsilon$ contains an element of degree $r$. So $s \leq r$.
Furthermore note that $N \subseteq \m^r F_0$. So $N_j = \m^j F_0 \cap N = N$ for $j \leq r$. It follows that $s \geq r$. Thus $s = r$.

(4) We first show that the complex is exact.
\begin{equation*}
\m^{i - s} F_1 \xrightarrow{\phi_{i-s}} \m^i F_1 \xrightarrow{\epsilon_i} \m^i M \rt 0 \quad \text{is exact for all} \ i \geq 0. 
\tag{**}
\end{equation*}
Here $\phi_{i-s}$ is the restriction of $\phi$ to $\m^{i-s}F_1$.
Notice $\ker \epsilon_i = N \cap \m^i F = N_i = \m^{i-s} N$ and the map $\phi_{i-s}$ maps  $\m^{i-s}F_1$ surjectively  to $\m^{i-s} N$.
We now tensor the exact sequence (**) with $A/\m$ to get the exact sequence
\[
\frac{\m^{i - s} F_1}{\m^{i + 1 - s} F_1}  \xrightarrow{\ov{\phi_{i-s}}} \frac{\m^{i} F_0}{\m^{i + 1 } F_0} \rt \frac{\m^{i}M}{\m^{i+1}M} \rt 0,
\]
for all $i \geq 0$. Thus we have an exact sequence
\begin{equation*}
G_\m(F_1)(-s) \xrightarrow{\ov{\phi}} G_\m(F_0) \rt G_\m(M) \rt 0. \tag{$\dagger$}
\end{equation*}
It can be easily verified that $\ov{\phi} = \ini(\phi)$. Furthermore as $G_\m(F_1) = R^a$ we get that $(\dagger)$ is part of a minimal resolution of $G_\m(M)$.

(5) This follows from the exact sequence (*) and (1).
\end{proof}
We now give 
\begin{proof}[Proof of Theorem \ref{first}]
We prove the result by induction on $p = \projdim M$. The result clearly holds when $p = 0$.  Furthermore when $p = 1$ the result follows from Lemma \ref{f-l}. We assume the result when $p = r \geq 1$ and prove it when $p = r + 1$.
Set $N = \Syz^A_1(M)$ and let $v(\phi_1) = s_1$. Then by Lemma \ref{f-l} we get
\[
\Syz^R_1(G_\m(M)) = G_\m(N)(-s_1).
\]
It follows that $G_\m(N)$ has a pure resolution. Truncate $\F$ to obtain a minimal resolution
$\F^\prime$ of $N$. By induction hypotheses $\ini(\F^\prime)$ is a minimal resolution of $G_\m(N)$. After shifting $\ini(\F^\prime)$ suitably and by Lemma \ref{f-l} we get that $\ini(\F)$ is a minimal resolution of $G_\m(M)$.
\end{proof}
\section{Proof of Theorem \ref{pure}}
In this section we give a proof of our main result.
We need the following  result due to Herzog-K\"{u}hl (see \cite[Theorem 1]{HK}) regarding modules having pure resolutions.
\begin{lemma}\label{CM-l}
Let $R = k[X_1,X_2, \cdots, X_n]$ and suppose $E$ is a finitely generated graded $R$-module with pure resolution of type
$(0,d_1,d_2,\cdots,d_p)$. Set $\beta_i = \beta_i(E)$. If for $i \geq 1$
\[
\beta_i = (-1)^{i+1}\beta_0 \prod_{j \neq i}\frac{d_j}{d_j-d_i},
\]
then $E$ is a \CM \ $R$-module.
\end{lemma}

We now give a proof of our main result:
\begin{proof}[Proof of Theorem \ref{pure}]
If $G_\m(M)$ has a pure resolution then by Theorem \ref{first} the complex $\ini(\mathbb{F})$ is a minimal resolution of $G_\m(M)$. The result now follow from 
\cite[p.\ 88]{BS}.

Conversely if $\ini(\mathbb{F})$ is acyclic and the betti -numbers satisfy the Herzog-K\"{u}hl conditions then by Lemma \ref{CM-l},  $E = \coker \ini(\phi_1) $ is \CM \ of dimension $n-p = \dim M$. Note we also have a surjective homomorphism $\epsilon \colon  G_\m(F_0) \rt G_\m(M)$ with $\epsilon \circ \ini(\phi_1) = 0$. So we have an exact sequence 
\[
0 \rt K \rt E \rt G_\m(M) \rt 0.
\] 
Note $\dim K \leq \dim E = \dim  G_\m(M)$.  As multiplicity of $E$ equals multiplicity of $G_\m(M)$ it follows that $\dim K < \dim E$. But $E$ is \CM. Therefore $K = 0$. So $G_\m(M) = E$ has a pure resolution.
\end{proof}
\section{An example}
All the computations  in this section were done using the computer algebra package SINGULAR. Let $Q = k[x, y, z]_\m$ where $\m = (x,y,z)$. Also $A = \widehat{Q}$ and $R = G_\m(A) = G_\m(Q) = k[x,y,z]$. Using SINGULAR one can verify that we have an exact sequence
\[
0 \rt Q^2 \xrightarrow{\psi} Q^4 \xrightarrow{\phi} Q^2 \rt M \rt 0,
\]
where
\[
\phi = \left( \begin{matrix} y^2 & x^2 & z^3 & 0 \\ 0 & 0 & z^2 & x^2  \end{matrix} \right) \quad \text{and} \quad \psi = \left( \begin{matrix} -x^2 & 0 \\ y^2 & z^3 \\0 & -x^2 \\ 0 & z^2 \end{matrix} \right). 
\]
Notice that the zeroth Fitting ideal of $M$ is 
\[
\Fitt_0(M) = I_2(\phi) = ( y^2z^2, y^2x^2, x^2z^2, x^4, z^3x^2).
\]
It is well-known that $\Fitt_0(M) \subseteq \ann(M)$ and that $\sqrt{\Fitt_0(M)} = \sqrt{\ann(M)}$. If $P$ is a prime containing $\Fitt_0(M)$ then clearly $P \supseteq (x,y)$ or $P \supseteq (x,z)$. Thus $\dim M = 1$. Therefore  $M$ is \CM.

Notice
\[
\ini(\phi) = \left( \begin{matrix} y^2 & x^2 & 0 & 0 \\ 0 & 0 & z^2 & x^2  \end{matrix} \right) \quad \text{and} \quad \ini(\psi) = \left( \begin{matrix} -x^2 & 0 \\ y^2 & 0 \\0 & -x^2 \\ 0 & z^2 \end{matrix} \right). 
\]
Using SINGULAR or the Buchsbaum-Eisenbud criterion it follows that  the complex
\begin{equation*}
0 \rt R(-4)^2 \xrightarrow{\ini(\psi)} R(-2)^4 \xrightarrow{\ini(\phi)} R^2 \rt 0 \tag{*}
\end{equation*}
is acyclic. 

Set $T = Q/(x^4, y^2z^2)$. Then $M$ is a \CM \ $T$-module. Also note that $T$ is \CM \ of dimension one with $e(T) = 16$. We have an exact sequence
\[
T^4 \xrightarrow{\ov{\phi}} T^2 \rt M \rt 0.
\]
Let $N = \Syz^T_1(M)$. Then using SINGULAR one can check that $e(N) = 24$. Thus $e(M) = 8$. Therefore by Theorem \ref{pure} we get that (*) is a minimal resolution of $G_\m(M)$.

\begin{remark}
Although in this example $p = 2$ and $\dim M = 1$, this is the first non-trivial example where we have an explicit resolution of $G_\m(M)$.
\end{remark}


\begin{thebibliography}{10}
\bibitem{BS}
M.~Boij and J.~S\"{o}derberg, 
\emph{Graded Betti numbers of Cohen-Macaulay modules and the multiplicity conjecture}, 
J. Lond. Math. Soc. (2) 78 (2008), no. 1, 85–-106. 

\bibitem{ES}
D.~Eisenbud, David and  F-O,~Schreyer, 
\emph{Betti numbers of graded modules and cohomology of vector bundles}, 
J. Amer. Math. Soc. 22 (2009), no. 3, 859–-888. 



\bibitem{HK}
J.~Herzog and M.~K\"{u}hl,
 \emph{On the Betti numbers of finite pure and linear resolutions},
  Comm. Algebra 12(1984) 1627–-1646.
  
  \bibitem{HM}
  C.~Huneke and M.~Miller, 
  \emph{A note on the multiplicity of Cohen–Macaulay algebras with pure resolutions},
  
  \bibitem{Mat}
H.~Matsumura, 
\emph{Commutative ring theory},
Translated from the Japanese by M. Reid. Second edition. Cambridge Studies in Advanced Mathematics, 8. Cambridge University Press, Cambridge, 1989. 
Canad. J. Math. 37 (1985) 1149–-1162
\end{thebibliography}
\end{document}